\documentclass{amsart}

\usepackage{amsfonts}
\usepackage{amsmath}
\usepackage{amsthm}

\usepackage{amsxtra}

\usepackage{hyperref}
 \usepackage[foot]{amsaddr}
\theoremstyle{plain}

\newtheorem{thm}{Theorem}[subsection]
\numberwithin{thm}{section}
\newtheorem{lem}[thm]{Lemma}
\newtheorem{prop}[thm]{Proposition}

\newtheorem{cor}[thm]{Corollary}
\theoremstyle{definition}
\newtheorem{defn}[thm]{Definition}

\theoremstyle{remark}
\newtheorem{rmk}[thm]{Remark}
\newtheorem{eg}[thm]{Example}

\makeatletter
\renewcommand{\email}[2][]{%
  \ifx\emails\@empty\relax\else{\g@addto@macro\emails{,\space}}\fi%
  \@ifnotempty{#1}{\g@addto@macro\emails{\textrm{(#1)}\space}}%
  \g@addto@macro\emails{#2}%
}
\makeatother

\begin{document}

\title{Finite Extensions of $\mathbb{Z}_\mathrm{max}$}
\author{Jeffrey Tolliver}
\email{tolliver@ihes.fr}
\address{Department of Mathematics, Johns Hopkins University\\3400 N Charles St\\ Baltimore, MD, USA 21218\newline
Present Address: Institut des Hautes \'Etudes Scientifiques\\35 Route de Chartres\\Bures-sur-Yvette France 91440}

\begin{abstract}We classify the semifields and division semirings containing the max-plus semifield  $\mathbb{Z}_\mathrm{max}$, which are finitely generated as $\mathbb{Z}_\mathrm{max}$-semimodules.
\end{abstract}

\maketitle
\section{Introduction}
There has been much interest recently in geometry over the tropical semifield $\mathbb{R}_\mathrm{max}=\mathbb{R}\cup\{\infty\}$, in which the addition operation is ${\max}$ and multiplication is given by the usual notion of addition\cite{trop1}.  In this paper, we will instead work with a related semifield $\mathbb{Z}_\mathrm{max}$, which is defined in a similar manner.

The semifield $\mathbb{Z}_\mathrm{max}$ has been studied by A. Connes and C. Consani in connection with the notion of the absolute point\cite{epicyclic}.   In particular, they have studied projective spaces over $\mathbb{Z}_\mathrm{max}$ and shown that they give a realization of J. Tits' ideas on a projective geometry over the "field with one element"\cite{tits}.  

A second motivation for studying $\mathbb{Z}_\mathrm{max}$ comes from the arithmetic site introduced by Connes and Consani in \cite{arithmeticsite}.  This site consists of the semiring $\mathbb{Z}_\mathrm{max}$ viewed as a sheaf of semirings over the topos of $\mathbb{N}^\times$-sets.  The Riemann zeta function counts the fixed $[0,1]_{\max}$-valued points of the Frobenius operator on the arithmetic site, analogous to the way that the zeta function of a function field counts the fixed points of the Frobenius operator on an algebraic curve.  In \cite{arithmeticsite} and \cite{arithmeticsite2}, Connes and Consani have made several steps towards mimicking the proof of the Riemann hypothesis for function fields in the setting of the arithmetic site.   Extensions of $\mathbb{Z}_\mathrm{max}$ play a key role in this theory because the points of the arithmetic site correspond to algebraic extensions.

A natural question that arises is to study the finite extensions of $\mathbb{Z}_\mathrm{max}$, that is semifields containing $\mathbb{Z}_\mathrm{max}$ which are finitely generated as a semimodule.  One reason for studying the finite extensions is geometric in nature.  When studying varieties over a non-algebraically closed field $K$, one needs to consider points with values not only in $K$, but also in finite extensions of $K$.  By analogy, one might expect that points with values in the extensions of $\mathbb{Z}_\mathrm{max}$ will be a necessary ingredient in developing a notion of algebraic geometry over $\mathbb{Z}_\mathrm{max}$ embodying Connes' and Consani's ideas about projective spaces over the field with one element.

In \cite{epicyclic}, Connes and Consani have discovered that for each $n>1$ there is a relative Frobenius map $\mathbb{Z}_\mathrm{max}\rightarrow \mathbb{Z}_\mathrm{max}$.  Furthermore they showed that this map gives a rank $n$ free semimodule $F^{(n)}$ over $\mathbb{Z}_\mathrm{max}$ which is a semifield.  The goal of this paper is to show that these are all of the finite extensions of $\mathbb{Z}_\mathrm{max}$.  

To each extension $L$ of $\mathbb{Z}_\mathrm{max}$, we may associate a group $L^\times/\mathbb{Z}_\mathrm{max}^\times$.  The key to understanding the finite extensions of $\mathbb{Z}_\mathrm{max}$ is corollary \ref{finiteuiandrank} which states that for every finite extension $L$ of $\mathbb{Z}_\mathrm{max}$ the group $L^\times/\mathbb{Z}_\mathrm{max}^\times$ is finite.  

Section \ref{basicdefsect} will give the basic definitions used throughout this paper.  In section \ref{basicdefsect} we will also classify finite extensions of the simplest idempotent semifield $\mathbb{B}$.

Section \ref{sect3} will introduce the notion of the unit index of an extension, which is the order of the group $L^\times/\mathbb{Z}_\mathrm{max}^\times$ associated to an extension $L$ of $\mathbb{Z}_\mathrm{max}$.  To show the theory of extensions with finite unit index is nontrivial, we will give a condition in which the unit index must be finite.  We will also show in theorem \ref{finiteuiclass} that for $n>1$, any extension of $\mathbb{Z}_\mathrm{max}$ has at most one subextension of a given unit index, and we will use this fact in corollary \ref{subextsofreals} to classify finite subextensions of $\mathbb{R}_\mathrm{max}/\mathbb{Z}_\mathrm{max}$.

Most of the results of section \ref{sect3} will be superseded by more general results in later sections.  Thus the reader may skip section 3 except for definition \ref{uidef} and the proof of theorem \ref{selectiveui}.  However section \ref{sect3} provides useful motivation for caring about whether an extension has finite unit index.

In section \ref{finitearchsect}, we will introduce the notion of an archimedean extension of an idempotent semifield.  Roughly speaking $L$ is archimedean  over $K$ if every element of $L$ is bounded above by an element of $K$ in a certain sense.  We will show that every finite archimedean extension of $\mathbb{Z}_\mathrm{max}$ has finite unit index.  

We would like to say that all finite extensions of $\mathbb{Z}_\mathrm{max}$ have finite unit index.  To do this we will show in sections \ref{convexsect} and \ref{maxarchsect} that every finite extension $L$ of any idempotent semifield $K$ is archimedean.  Then the results of section \ref{finitearchsect} will apply.  The strategy to proving this will involve constructing the maximal archimedean subextension $L_\mathrm{arch}$ of the extension $L$ over $K$.  Section \ref{convexsect} is devoted to  introducing a notion of convexity that will allow us to prove in section \ref{maxarchsect} that $L=L_\mathrm{arch}$.  This will imply that $L$ is archimedean.

In section \ref{classificationsect}, we will classify extensions of $\mathbb{Z}_\mathrm{max}$ with finite unit index, by showing in theorem \ref{classificationui} that they are all $F^{(n)}$ for some $n$. Since all finite extensions of $\mathbb{Z}_\mathrm{max}$ have finite unit index, this gives us a classification of the finite extensions.  

Suppose $L$ has finite unit index over $\mathbb{Z}_\mathrm{max}$.  The first step to showing that $L\cong F^{(n)}$ will be to study the structure of the multiplicative group $L^\times$, which we will see to be isomorphic to $\mathbb{Z}$.  To understand the addition, we show that the embedding $\mathbb{Z}_\mathrm{max}\rightarrow L$ tells us how to add $n$th powers.  We then show that this completely determines the additive structure by using lemma \ref{monotonicdivision}, which states that the $n$th root operation is monotonic in a suitable sense.

After studying the finite extensions, in sections \ref{dsasect} and \ref{fdsasect} we will outline how these results may be generalized to the noncommutative case of division semialgebras over $\mathbb{Z}_\mathrm{max}$.

\section{Basic Definitions and Examples}\label{basicdefsect}

\begin{defn}A (commutative) semiring $R$ is a set together with 2 binary operations (called addition and multplication) such that $R$ is a commutative monoid under each operation and the distributive law holds.  It is idempotent if $x+x=x$ for all $x\in R$.  It is selective if for all $x,y\in R$ one has either $x+y=x$ or $x+y=y$.  A semifield is a semiring $R$ in which all nonzero elements are units.\end{defn}

\begin{eg}Let $\mathbb{B}=\{0,1\}$ in which addition is given by $x+0=0+x=0$ for all $x$ and $1+1=1$, and with the obvious notion of multiplication.  Then $\mathbb{B}$ is an idempotent semifield.  More generally let $M$ be a totally ordered abelian group.  Then $M_\mathrm{max}=M\cup\{-\infty\}$ is an idempotent semiring in which addition is $\max$ and multiplication is the group operation of $M$.  Then $\mathbb{B}=M_\mathrm{max}$ where $M$ is the trivial group.\end{eg}

\begin{rmk}\label{elementu} There is an element $u\in \mathbb{Z}_\mathrm{max}$  such that $\mathbb{Z}_\mathrm{max}=\{0\}\cup \{u^n\mid u\in\mathbb{Z}\}$ and $u+1=u$.  We will write elements of $\mathbb{Z}_\mathrm{max}$ this way to avoid the ambiguity between addition in $\mathbb{Z}$ and in $\mathbb{Z}_\mathrm{max}$.\end{rmk}

\begin{defn}An extension $L$ of a semifield $K$ consists of a semifield $L$ and an injective homomorphism $K\rightarrow L$.  The extension is finite if the homomorphism makes $L$ into a finitely generated semimodule.\end{defn}

\begin{eg}\label{totord}Let $M$ be a totally ordered abelian group and $N\subseteq M$ be a subgroup.  Then $M_\mathrm{max}$ is an extension of $N_\mathrm{max}$.\end{eg}

\begin{eg}\label{Fn}Fix a positive integer $n$.  Define a map $\mathbb{Z}_\mathrm{max}\rightarrow\mathbb{Z}_\mathrm{max}$ sending each nonzero element $u^k$ to $u^{nk}$ and sending $0$ to $0$.  Then this homomorphism is injective, so gives an extension which will be denoted $F^{(n)}$.  It is easily checked that $1,u,\ldots,u^{n-1}$ generate $F^{(n)}$ as a semimodule over $\mathbb{Z}_\mathrm{max}$\footnote{In fact this is a minimal set of generators, so $F^{(n)}$ is a rank $n$ semimodule over $\mathbb{Z}_\mathrm{max}$.}, so the extension is finite.
\end{eg}

We will conclude this section by classifying finite extensions of $\mathbb{B}$.  To do this we will need two lemmas. The first of these two lemmas can be obtained by translating a standard result on lattice ordered groups into the language of idempotent semifields.  However, we will give a different, and hopefully simpler, proof.
\begin{lem}\label{monotonicdivision} Let $K$ be an idempotent semifield.  Let $x,y\in K$ be such that $x^n+y^n=y^n$ for some $n>0$.  Then $x+y=y$.\end{lem}

\begin{proof}We may assume $x,y\neq 0$.  Then $x+y\neq 0$.  We compute $(x+y)^n=x^n+x^{n-1}y+\ldots+xy^{n-1}+y^n=x^{n-1}y+\ldots+xy^{n-1}+y^n
=y(x^{n-1}+x^{n-2}y+\ldots+xy^{n-2}+y^{n-2})=y(x+y)^{n-1}$.  Dividing by $(x+y)^{n-1}$ gives $x+y=y$.\end{proof}

\begin{lem}\label{rootsofunity}Let $K$ be an idempotent semifield, and $x\in K$ be a root of unity.  Then $x=1$.\end{lem}
\begin{proof}For some $n$, we have $x^n=1$. By lemma \ref{monotonicdivision}, it follows that $x+1=1$.

$x^{-1}$ is also a root of unity so lemma \ref{monotonicdivision} gives $x^{-1}+1=1$.  Hence $x+1=x$.  By transitivity of equality we have $x=1$.\end{proof}

\begin{thm}\label{Balgclosed}Let $L$ be a finite extension of the idempotent semifield $\mathbb{B}$.  Then $L=\mathbb{B}$.\end{thm}
\begin{proof}Since $L$ is finitely generated as a semimodule over $\mathbb{B}$ and $\mathbb{B}$ is finite, it follows that $L$ is finite.  Then $L^\times$ is a finite group and hence is torsion.   By lemma \ref{rootsofunity}, $L^\times=\{1\}$.  Hence $L=\mathbb{B}$.\end{proof}

\section{Finite subextensions of $\mathbb{R}_\mathrm{max}$ over $\mathbb{Z}_\mathrm{max}$}\label{sect3}

In this section we will associate a number called the unit index to any extension of semifields.  As an application, and as motivation for the approach of later sections, we will classify finite subextensions of the infinite extension $\mathbb{R}_\mathrm{max}$ over $\mathbb{Z}_\mathrm{max}$.  The first step will be to show in theorem \ref{selectiveui} that the finite subextensions have finite unit index.  We will then study the subextensions of $\mathbb{R}_\mathrm{max}$ with finite unit index by relating them to finite subgroups of the circle group $\mathbb{R}/\mathbb{Z}$.

\begin{defn}\label{uidef}Let $L$ be an extension of a semifield $K$.  We define the unit index of the extension to be $\mathrm{ui}(L/K)=|L^\times/K^\times|$.\end{defn}

\begin{eg}Pick $v\in F^{(n)}$ such that $F^{(n)}=\{0\}\cup \{v^k\ \mid k\in\mathbb{Z}\}$.  Then $\mathbb{Z}_\mathrm{max}=\{0\}\cup \{v^{kn}\}$.  Then $(F^{(n)})^\times$ is cyclic with generator $v$ while $\mathbb{Z}_{\mathrm{\max}}^{\times}$ is cyclic with generator $v^n$.  It is easily seen that $\mathrm{ui}(F^{(n)}/\mathbb{Z}_\mathrm{max})=n$.\end{eg}

\begin{defn}A idempotent semigroup $M$ is selective if for all $x,y\in M$ either $x+y=x$ or $x+y=y$.\end{defn}

Of course $\mathbb{R}_\mathrm{max}$ is selective, as is any subsemimodule of $\mathbb{R}_\mathrm{max}$. This property will make it easy to show in the following thoerem that the finite subextensions of $\mathbb{R}_\mathrm{max}$ over $\mathbb{Z}_\mathrm{max}$ have finite unit index. 

\begin{thm}\label{selectiveui}Let $L$ be a finite extension of $\mathbb{Z}_\mathrm{max}$ in which $L$ is selective.  Then $\mathrm{ui}(L/\mathbb{Z}_\mathrm{max})<\infty$.\end{thm}
\begin{proof}Note that because $L$ is selective, every subset is closed under addition.  Let $S$ be a finite set generating $L$ as a semimodule over $\mathbb{Z}_\mathrm{max}$.  Without loss of generality, we may assume $0\not\in S$.  $S\mathbb{Z}_\mathrm{max}$ is a subsemimodule of $L$ over $\mathbb{Z}_\mathrm{max}$ because it is closed under scalar multiplation by construction, and because it is closed under addition.  Since $S\subseteq S\mathbb{Z}_\mathrm{max}$, one has $L=S\mathbb{Z}_\mathrm{max}$.  Then $L^\times=S\mathbb{Z}$, and $S$ surjects onto $L^\times/\mathbb{Z}$.  Hence $|L^\times/\mathbb{Z}|\leq |S|<\infty$.\end{proof}

We will see in theorem \ref{finiteuiandrank} that the above theorem holds without the hypothesis that $L$ is selective.  However it will take several sections to develop the machinery necessary to drop this hypothesis. 

For an extension $E$ of $\mathbb{Z}_\mathrm{max}$, it will be helpful to understand the group structure of the quotient group $E^\times/\mathbb{Z}$.  To do this, we will need the following standard lemma.

\begin{lem}\label{ordercount}Let $G$ be a group.  Suppose that for all $n\in\mathbb{N}$, $G$ has at most $n$ elements of order dividing $n$.  Then every finite subgroup of $G$ is cyclic, and there is at most one finite subgroup of a given order.\end{lem}

We will make use of the following corollary with $M=E^\times$.

\begin{cor}\label{onesubgroup}Let $M$ be a torsionfree abelian group.  Let $\mathbb{Z}\subseteq M$ be an infinite cyclic subgroup.  Then for each positive integer $n$, $M/\mathbb{Z}$ has at most one subgroup of order $n$ and all finite subgroups are cyclic.\end{cor}
\begin{proof}Let $n$ be a positive integer.  By lemma \ref{ordercount} it suffices to show that $M/\mathbb{Z}$ has at most $n$
elements of order dividing $n$.    Let $\bar{x}\in M/\mathbb{Z}$ have order dividing $n$ and let $\hat{x}\in M$ be any lift.  Then $n\hat{x}\in\mathbb{Z}$, and there exists $k\in\mathbb{Z}$ such that $n(\hat{x}-k)\in \{0,1,\ldots,n-1\}$.  Let $x=\hat{x}-k$, which is also a lift of $\bar{x}$ to $M$.  Since $M$ is torsionfree, each equation $nt=m$ with $n,m\in\mathbb{Z}$ has at most one solution $t$.  Since there are $n$ possibilities for $nx$, there are at most $n$ choices for $x$ and hence for $\bar{x}$.\end{proof}

The following theorem is the first hint that the unit index will be relevant to the problem of classifying finite extensions of $\mathbb{Z}_\mathrm{max}$.  Furthermore it will allow us to easily classify those finite extensions which are contained inside $\mathbb{R}_\mathrm{max}$.

\begin{thm}\label{finiteuiclass}Let $E$ be an extension of $\mathbb{Z}_\mathrm{max}$.  Let $n$ be a positive integer.  Then there is at most one subextension $L$ of $E$ such that $\mathrm{ui}(L/\mathbb{Z}_\mathrm{max})=n$.\end{thm}
\begin{proof}Let $A$ be the set of all subextensions of $E$ over $\mathbb{Z}_\mathrm{max}$.  Let $B$ be the set of subgroups of $E^\times$ containing $\mathbb{Z}_\mathrm{max}^\times=\mathbb{Z}$.  Define a map $\phi:A\rightarrow B$ by $\phi(L)=L^\times$.  If $\phi(L)=\phi(M)$, then $L=\{0\}\cup L^\times=\{0\}\cup M^\times=M$, so $\phi$ is injective.  Let $C$ be the set of subgroups of $E^\times/\mathbb{Z}_\mathrm{max}^\times$.  The fourth isomorphism theorem states that the map $\psi:B\rightarrow C$ given by $\psi(G)=G/\mathbb{Z}_\mathrm{max}^\times$ is a bijection.  Hence the map $A\rightarrow C$ sending $L$ to $L^\times/\mathbb{Z}_\mathrm{max}^\times$  is injective.

This map clearly restricts from an injection from the set of subextensions with unit index $n$ to the set of subgroups of $E^\times/\mathbb{Z}_\mathrm{max}^\times=E^\times/\mathbb{Z}$ with order $n$.  By lemma \ref{rootsofunity} and corollary \ref{onesubgroup}, there is at most one such subgroup.  Hence there is at most one subextension with unit index $n$.\end{proof}

\begin{rmk}Suppose $E$ is selective.  Then if $G$ is a subgroup of $E^\times$ then $\{0\}\cup G$ is a subsemifield of $E$; it is closed under addition because every subset of a selective semigroup is closed under addition.  Since $\phi(\{0\}\cup G=G$, the map $\phi$ from the proof of theorem \ref{finiteuiclass} is bijective in this case.  Hence there is a bijective correspondence between subextensions of $E$ over $\mathbb{Z}_\mathrm{max}$ and subgroups of $E^\times/\mathbb{Z}$.  \end{rmk}

\begin{cor}\label{subextsofreals}Let $L$ be a finite subextension of $\mathbb{R}_\mathrm{max}$ over $\mathbb{Z}_\mathrm{max}$.  Then there exists $n$ such that $L=(\frac{1}{n}\mathbb{Z})_\mathrm{max}$\footnote{This is the semifield associated to the totally ordered subgroup $\frac{1}{n}\mathbb{Z}\subseteq \mathbb{R}$ via example \ref{totord}.  One can easily exhibit an explicit isomorphism of extensions $(\frac{1}{n}\mathbb{Z})_\mathrm{max}\cong F^{(n)}$.  If we identify $F^{(n)}$ with $\mathbb{Z})_\mathrm{max}$ as in example \ref{Fn}, this isomorphism sends $\frac{a}{n}$ to $u^a$.}.
\end{cor}
\begin{proof}Since $L\subseteq\mathbb{R}_\mathrm{max}$, $L$ is selective.  By theorem \ref{selectiveui}, $L$ has finite unit index.  Let $n=\mathrm{ui}(L/\mathbb{Z}_\mathrm{max})$.  Then $(\frac{1}{n}\mathbb{Z})_\mathrm{max}$ has unit index $n$ over $\mathbb{Z}_\mathrm{max}$.  By theorem \ref{finiteuiclass} they are equal.\end{proof}

\section{Finite archimedean extensions of $\mathbb{Z}_\mathrm{max}$}\label{finitearchsect}

In this section, we will give a criterion that is useful for proving an extension has finite unit index.  In later sections, we will use this criterion to prove that every finite extension of $\mathbb{Z}_\mathrm{max}$ has finite unit index.

\begin{defn}Let $K$ be an idempotent semifield.  An extension $L$ over $K$ is called archimedean if for all $x\in L$, there exists $y\in K$ such that $x+y=y$.\end{defn}

The terminology comes from the following example.

\begin{eg}$\mathbb{R}_\mathrm{max}$ can be seen to be an archimedean extension of $\mathbb{Z}_\mathrm{max}$.  This is because of the archimedean property of the real numbers, which states that for every $x\in\mathbb{R}$ there exists $n\in\mathbb{Z}$ such that $x\leq n$ or equivalently $\max{x,n}=n$.  \end{eg}

\begin{lem}\label{archlowerbound}Let $L$ be an archimedean extension of an idempotent semifield $K$.  Then for all nonzero $x\in L$ there exists nonzero $z\in K$ such that $x+z=x$.\end{lem}
\begin{proof}There is some $y\in K$ such that $x^{-1}+y=y$, which is clearly nonzero.  After multiplying by $xy^{-1}$, we get $y^{-1}+x=x$, so we may take $z=y^{-1}$.\end{proof}

For the remainder of this section, let $L$ be finite and archimedean over $\mathbb{Z}_\mathrm{max}$, and let $S\subseteq L$ be a finite set which generates $L$ as a $\mathbb{Z}_\mathrm{max}$-semimodule.  We may assume $0\not\in S$.  The goal for the remainder of the section will be to show that $\mathrm{ui}(L/\mathbb{Z}_\mathrm{max})<\infty$.   If we can show that $S\mathbb{Z}_\mathrm{max}=\{sx\mid s\in S,\,x\in\mathbb{Z}_\mathrm{max}\}$ is closed under addition, then we can apply the proof of theorem \ref{selectiveui} to prove theorem \ref{archui}.  Unfortunately, there is no reason to believe that it is closed under addition.\footnote{In the case $L=F^{(n)}$, one can show that $L=S\mathbb{Z}_\mathrm{max}$.  The classification theorem that we are working towards will then imply that $S\mathbb{Z}_\mathrm{max}$ is always closed under addition.  However, we do not know a direct way to show that $S\mathbb{Z}_\mathrm{max}$ is already closed under addition without enlarging $S$.}  However, we will see that we can construct a larger, but still finite, generating set $T$ such that $T\mathbb{Z}_\mathrm{max}$ is closed under addition.

\begin{lem}\label{Mlemma}Let $S$ be as above and let $S^{-1}S=\{s_1^{-1}s_2\mid s_1,s_2\in S\}$.  There exists $M\in \mathbb{Z}$ such that $x+u^M=u^M$ and $x+u^{-M}=x$ for all $x\in S^{-1}S$.  Furthermore, any number larger than $M$ also has this property\end{lem}
\begin{proof}Note that if $m>n$ and $x+u^n=u^n$ then $x+u^m=x+u^m+u^n=u^m+u^n=u^m$.  Similarly if $x+u^{-n}=x$ and $m>n$ then $x+u^{-m}=x$.  Since $S^{-1}S$ is finite, these remarks allow us to construct a different value of $M$ for each of the statements, and take the maximum of all of them.  Let $x\in S^{-1}S$.  Then since $L$ is archimedean over $\mathbb{Z}_\mathrm{max}$, there exists $M$ such that $x+u^M=u^M$.  By lemma \ref{archlowerbound}, there exists $M$ such that $x+u^{-M}=x$.\end{proof}

For the remainder of this section we will let $M$ be the value constructed in the previous lemma. 

 Let $T_n=\{s+\displaystyle\sum_{i=1}^{n} u^{k_i}s_i \mid s,s_1,\ldots,s_n\in S,\, k_1,\ldots,k_n\in \{-M,\ldots,0\}\}$.  Let $T=\bigcup_{n\geq 0} T_n$.

\begin{lem}$T_n\subseteq T_{n+1}$ for all $n$.\end{lem}
\begin{proof}Let $s+\displaystyle\sum_{i=1}^{n} u^{k_i}s_i \in T_n$.  $s+\displaystyle\sum_{i=1}^{n} u^{k_i}s_i=s+s+\displaystyle\sum_{i=1}^{n} u^{k_i}s_i+u^{k_n}s_n\in T_{n+1}$.\end{proof}

\begin{lem}\label{finiteT}Let $N=(M+1)|S|$.  Then $T=T_N$, and $T$ is finite.\end{lem}

\begin{proof}It suffices to show for each $n$ that $T_n\subseteq T_N$.  We know this in the case where $n\leq N$.  For $n>N$, we proceed by induction.  Let $s+\displaystyle\sum_{i=1}^{n} u^{k_i}s_i\in T_n$.  Since there are $M+1$ choices for $k_i$, and $|S|$ choices for $s_i$,  the pigeon hole principle implies some term is repeated.  Since addition is idempotent, we can remove the repeated term, so $s+\displaystyle\sum_{i=1}^{n} u^{k_i}s_i\in T_{n-1}$.  By the inductive hypothesis, $T_{n-1}\subseteq T_N$, so $T_n\subseteq T_N$.  It is clear that $T_N$ is finite; in fact for any $n$, $T_n$ has at most $|S|^{n+1} (M+1)^n$ elements.\end{proof}

Since $S\subseteq T$, $T$ is also a finite generating set for $L$.  The next step is to show that $T$ is closed under addition.

\begin{lem}\label{negTlemma}Let $x=s+\displaystyle\sum_{i=1}^{n} u^{k_i}s_i$ for some $s,s_1,\ldots,s_n$ where $k_1,\ldots,k_n$ are nonpositive integers.  Then $x\in T$.\end{lem}
\begin{proof}Suppose $k_i<-M$.  Then by lemma \ref{Mlemma}, $s_i^{-1}s+u^{k_i}=s_i^{-1}s$.  Hence $s+u^{k_i}s_i=s$, so we may drop the term $u^{k_i}s_i$.  After dropping all such terms, we may suppose without loss of generality that $k_i\geq M$ for all $i$.  But then we trivially have $s+\displaystyle\sum_{i=1}^{n} u^{k_i}s_i\in T$.\end{proof}

\begin{lem}Let $n\geq 1$.  Let $z=\displaystyle\sum_{i=1}^{n} u^{k_i}s_i$ with $s_i\in S$ and $k_i\in \mathbb{Z}$.  Then $z\in T\mathbb{Z}_\mathrm{max}$.  Conversely every nonzero element of $T\mathbb{Z}_\mathrm{max}$ has this form for some $n$.\end{lem}

\begin{proof}After rearranging terms, we may suppose without loss of generality that $k_n\geq k_i$ for all $i$.   Then $u^{-k_n}z=s_n+\displaystyle\sum_{i=1}^{n-1} u^{k_i-k_n}s_i$.  By lemma \ref{negTlemma}, $u^{-k_n}z\in T$.  Hence $z\in T\mathbb{Z}_\mathrm{max}$.

The converse is trivial.\end{proof}

\begin{cor}\label{Tadditivelyclosed}$T\mathbb{Z}_\mathrm{max}$ is closed under addition.\end{cor}

In what follows, the next theorem will play a similar role to that played by theorem \ref{selectiveui} in section \ref{sect3}.  We will later see that all finite extensions are archimedean, and so this theorem is much more general than it would first appear.

\begin{thm}\label{archui}Let $L$ be a finite archimedean extension of $\mathbb{Z}_\mathrm{max}$.  Then $\mathrm{ui}(L/\mathbb{Z}_\mathrm{max})<\infty$.\end{thm}

\begin{proof}Let $S$ generate $L$ as a semimodule.  Let $T$ be the set defined earlier in this section.   Since $S\subseteq T$, $T$ also generates $L$.  By lemma \ref{finiteT}, $T$ is finite.  By corollary \ref{Tadditivelyclosed}, $T\mathbb{Z}_\mathrm{max}$ is closed under addition.  One can apply the proof of theorem \ref{selectiveui} to show that $T$ surjects onto $L^\times/\mathbb{Z}_\mathrm{max}^\times$.  The result follows.\end{proof}

\section{Convex subsemifields}\label{convexsect}

In this section we introduce the notion of a convex subsemifield of an idempotent semifield.  A convex subsemifield $K\subseteq L$ will have the property that addition in $L/K^\times$ is well-defined.  We will use this property to constrain the possible subextensions of the extension $L$ of $K$.

The following definition is essentially the same as the definition of a convex $\ell$-subgroup given in \cite{lattice}.

\begin{defn}Let $L$ be an idempotent semifield.  A subsemifield $K\subseteq L$ is called convex if for any $x\in L$ such that there exist $y,z\in K$ with $x+y=y$ and $x+z=x$, one has $x\in K$.\end{defn}

\begin{eg}Give to $\mathbb{Z}\times \mathbb{Z}$ the lexicographical order, in which $(a,b)\leq (x,y)$ if $a<x$ or if $a=x$ and $b\leq y$.  Identify $\mathbb{Z}$ with a subgroup of $\mathbb{Z}\times \mathbb{Z}$ by identifying $n$ with $(0,n)$.  Then $\mathbb{Z}_\mathrm{max}\subseteq(\mathbb{Z}\times \mathbb{Z})_\mathrm{max}$ is a convex subsemifield.  This follows from the fact that the inequalities $(0,a)\leq (m,n)\leq (0,b)$ imply $m=0$, and the fact that $x\leq y$ if and only if $\max(x,y)=y$.\end{eg}

If $K\subseteq L$ is a convex subsemifield, we consider an equivalence relation $\sim$ on $L$ given by $x\sim y$ if there exists $u\in K^\times$ with $x=uy$.  We denote the quotient by $L/K^\times$.

\begin{thm}\label{convexquotient}\cite[2.2.1]{lattice}Let $L$ be an idempotent semifield, and $K$ be a convex subsemifield.  Then $L/K^\times$ is an idempotent semifield.\end{thm}

\begin{proof}The only thing to check is that addition is well defined.  Let $x,y\in L$ and $u\in K$.  We must show that $x+y\sim x+uy$.  Equivalently we must show $z\in K$ where $z=(x+y)^{-1}(x+uy)$.

Suppose $u+1=u$.  Then $ux+x=ux$.  Hence $u(x+y)+(x+uy)=u(x+y)$.  Then $u+z=u$.  Also $uy+y=uy$ so $(x+uy)+(x+y)=x+uy$.  Hence $z+1=z$.  Since $1,u\in K$, it follows from convexity that $z\in K$.  Hence $x+y\sim x+uy$.

In general, we have $(u+1)+1=u+1$, so $x+y\sim x+(u+1)y$ and it suffices to show that $x+uy\sim x+(u+1)y$.  Equivalently, it suffices to show that $u^{-1}x+y\sim u^{-1}x+(1+u^{-1}y)$.  But this follows from the case already considered since $(1+u^{-1})+1=1+u^{-1}$.
  \end{proof}

\begin{thm}\label{convexthm}Let $E$ be an extension of an idempotent semifield $K$. Suppose $K\subseteq E$ is convex.  Then the extension $E$ over $K$ has no nontrivial finite subextensions.\end{thm}
\begin{proof}Let $L$ be a finite subextension of $E$ over $K$.  Then $K$ is convex in $L$.  Since $L$ is a finite extension, there is a finite set $S$ such that every element $x\in L$ can be written as a finite sum $x=\sum a_i s_i$ for $a_i\in K$ and $s_i$ in $S$.  Then every element of $L/K^\times$ can be written as a finite sum $\bar{x}=\sum \bar{a_i}\bar{s_i}$ where $\bar{a_i}\in K/K^\times=\mathbb{B}$ and $\bar{s_i}$ ranges over a finite set $\bar{S}$.  Hence $L/K^\times$ is a finite extension of $\mathbb{B}$.  By theorem \ref{Balgclosed}, $L/K^\times=\mathbb{B}$.  Hence for all $x\in L$, one has $x=0$ or $x\in K^\times$.  It follows that $L=K$.\end{proof}

\section{The maximal archimedean subextension}\label{maxarchsect}
When thinking about archimedean subextensions of a given extension, a natural question that arises is whether there is a maximal archimedean subextension, which contains every other archimedean subextension.  In this section we will explicitly construct this maximal archimedean subextension.  Applying the results of section \ref{convexsect} in this context will imply all finite extensions are archimedean, and so we may drop the archimedean hypothesis from theorem \ref{archui}.

\begin{defn}\label{maxarchdef}Let $L$ be an extension of an idempotent semifield $K$.  We define $L_\mathrm{arch}=\{x\in L|\ x+y=y,x+z=x\textrm{ for some }z,y\in K\}$.\end{defn}

\begin{lem}$L_\mathrm{arch}$ is a subsemifield of $L$ and contains $K$.\end{lem}
\begin{proof}Let $x_1,x_2\in L_\mathrm{arch}$.  Then there exists $y_1,y_2,z_1,z_2\in K$ such that $x_1+y_1=y_1$, $x_2+y_2=y_2$, $x_1+z_1=x_1$, and $x_2+z_2=x_2$.  Then $(x_1+x_2)+(y_1+y_2)=y_1+y_2$ and $(x_1+x_2)+(z_1+z_2)=x_1+x_2$.   Thus $x_1+x_2\in L_\mathrm{arch}$.  

Also $x_1x_2+y_1y_2=x_1x_2+(x_1+y_1)(x_2+y_2)=x_1x_2+y_1x_2+x_1y_2+y_1y_2=(x_1+y_1)(x_2+y_2)=y_1y_2$. A similar computation shows $x_1x_2+z_1z_2=x_1x_2$.  Thus $x_1x_2\in L_\mathrm{arch}$.  The rest of the proposition is trivial.
\end{proof}

\begin{prop}Let $L$ be an extension of an idempotent semifield $K$.  $L_\mathrm{arch}$ is the maximal archimedean subextension of $L$; In other words, it is an archimedean subextension and every other archimedean subextension is contained in it.\end{prop}
\begin{proof}
By definition, for every $x\in L_\mathrm{arch}$, there exists $y\in K$ such that $x+y=y$.

For the converse let $F$ be an archimedean subextension of $L$ over $K$.  Let $x\in F$.  Then there exists $y$ such that $x+y=y$.  Since $x^{-1}\in F$, there exists a nonzero element $z^{-1}\in K$ such that $x^{-1}+z^{-1}=z^{-1}$ so $x+z=x$.  Since $x\in L$, the above equalities show $x\in L_\mathrm{arch}$.  Hence $F\subseteq L_\mathrm{arch}$.\end{proof}

\begin{thm}$L_\mathrm{arch}$ is a convex subsemifield of $L$.\end{thm}
\begin{proof}
Let $x\in L$.  Suppose there exist $y,z\in L_\mathrm{arch}$ such that $x+y=y$ and $x+z=x$.  By the definition of $L_\mathrm{arch}$, there exist $y',z'\in K$ such that $y+y'=y'$ and $z+z'=z$.  Then $x+z'=(x+z)+z'=x+z=x$ and $x+y'=x+(y+y')=y+y'=y'$.   Hence $x\in L_\mathrm{arch}$.
\end{proof}
\begin{cor}\label{finitearchimedean}Every finite extension $L$ over an idempotent semifield $K$ is archimedean.\end{cor}
\begin{proof}$L$ is a finite extension over $L_\mathrm{arch}$ with $L_\mathrm{arch}$ convex inside $L$.  By theorem \ref{convexthm}, $L=L_\mathrm{arch}$.  Hence $L$ is archimedean over $K$.\end{proof}

We can now prove the following generaliztion of theorems \ref{selectiveui} and \ref{archui}
\begin{cor}\label{finiteuiandrank}Let $L$  be a finite extension of $\mathbb{Z}_\mathrm{max}$.  Then $\mathrm{ui}(L/\mathbb{Z}_\mathrm{max})<\infty$.\end{cor}
\begin{proof}Use corollary \ref{finitearchimedean} and theorem \ref{archui}\end{proof}

\section{The classification theorem}\label{classificationsect}

In this section, we will finally prove the classification of finite extensions of $\mathbb{Z}_\mathrm{max}$.

The following lemma is a consequence of the classification of finitely generated abelian groups.

\begin{lem}\label{Zses}Let $M$ be a torsion free abelian group, and $N$ be a finite abelian group.  Suppose there is a short exact sequence $0\rightarrow\mathbb{Z}\rightarrow M\rightarrow N\rightarrow 0$.  Then $M\cong\mathbb{Z}$.\end{lem}

\begin{thm}\label{classificationui}Let $L$ be an extension of $\mathbb{Z}_\mathrm{max}$ with $\mathrm{ui}(L/\mathbb{Z}_\mathrm{max})<\infty$.  Then $L\cong F^{(n)}$ as extensions of $\mathbb{Z}_\mathrm{max}$ for some $n$.\end{thm}
\begin{proof}Fix an element $u\in\mathbb{Z}_\mathrm{max}$ as in remark \ref{elementu}.

We have a short exact sequence $0\rightarrow \mathbb{Z}_\mathrm{max}^\times\rightarrow L^\times \rightarrow L^\times/\mathbb{Z}_\mathrm{max}^\times\rightarrow 0$.  $L^\times$ is torsionfree by lemma \ref{rootsofunity}.  By assumption, $L^\times/\mathbb{Z}_\mathrm{max}^\times$ is finite.  By lemma \ref{Zses} $L^\times\cong \mathbb{Z}$.  Pick a generator $v$ of $L^\times$.  Then $L=\{0\}\cup \{v^k\mid k\in\mathbb{Z}\}$.  Since $u\in\mathbb{Z}_\mathrm{max}\subseteq L$ is nonzero there exists $n\neq 0$ such that $u=v^n$.  By picking the other generator of $L^\times$ if neccessary, we may assume without loss of generality that $n>0$.

To determine the addition in $L$, it suffices to compute $v^a+v^b$ for $a,b\in\mathbb{Z}$.  We may suppose without loss of generality that $a>b$.  Then $(v^a)^n+(v^b)^n=u^a+u^b=u^a=(v^a)^n$.  By lemma \ref{monotonicdivision}, $v^a+v^b=v^a$. 

Hence $L\cong \mathbb{Z}_\mathrm{max}$ under the map sending $v$ to $u$.  Then the extension $L$ of $\mathbb{Z}_\mathrm{max}$ may be identified with the extension given by the composite map $\mathbb{Z}_\mathrm{max}\rightarrow L\cong\mathbb{Z}_\mathrm{max}$ sending $u$ to $u^n$.  But this extension is $F^{{n}}$.\end{proof}

Combining theorem \ref{classificationui} and corollary \ref{finiteuiandrank} gives us the following classification of finite extensions of $\mathbb{Z}_\mathrm{max}$.

\begin{thm}Let $L$ be a finite extension of $\mathbb{Z}_\mathrm{max}$.  Then $L\cong F^{(n)}$ as extensions of $\mathbb{Z}_\mathrm{max}$.\end{thm}

\section{Division semialgebras with finite unit index}\label{dsasect}
Unlike the previous sections, throughout this section, we will use the term semiring to refer to a possibly noncommutative semiring.

\begin{defn}A division semialgebra over a semifield $K$ is a division semiring $D$ together with an injective homomorphism from $K$ to the center of $D$.  It is finite if $D$ is finite as a left semimodule over $K$.\end{defn}

We define the unit index of a division semialgebra analogously to definition \ref{uidef}.

\begin{lem}\label{noncommmonotonicdivision}Let $D$ be an idempotent division semiring.  Let $x,y\in D$ satisfy $xy=yx$.  Suppose $x^n+y^n=y^n$ for some $n\geq 1$.  Then $x+y=y$\end{lem}
\begin{proof}This can be proven as in lemma \ref{monotonicdivision}\end{proof}

Lemma \ref {noncommmonotonicdivision} provides us with the following analogues of lemma \ref{rootsofunity} and theorems \ref{Balgclosed}.
\begin{cor}Let $D$ be an idempotent division semiring.  Then $D^\times$ is torsion free.\end{cor}

\begin{proof}Let $x\in D^\times$ be torsion of order $n$.  Since $x^n+1=1$, $x+1=1$.  Similarly $x+1=x$, so $x=1$.\end{proof}

\begin{cor}\label{ncBalgclosed}Let $D$ be a finite division semialgebra over $\mathbb{B}$.  Then $D=\mathbb{B}$.\end{cor}

\begin{thm}Let $D$ be a division semialgebra over $\mathbb{Z}_\mathrm{max}$ with finite unit index.  Then $D$ is selective.\end{thm}
\begin{proof}Let $x,y\in D$.  We wish to show either $x+y=x$ or $x+y=y$.  If either of $x$ or $y$ is zero, we are done.  It suffices to show $xy^{-1}+1=xy^{-1}$ or $xy^{-1}+1=1$.  In other words, we can assume without loss of generality that $y=1$.

Let $n=\mathrm{ui}(D/\mathbb{Z}_\mathrm{max})$.  Then by Lagrange's theorem $x^n\in \mathbb{Z}_\mathrm{max}$.  Since $\mathbb{Z}_\mathrm{max}$ is selective, $x^n+1=1$ or $x^n+1=x^n$.  Since $x$ commutes with $1$, we may apply lemma \ref{noncommmonotonicdivision} to see that $x+1=1$ or $x+1=x$.\end{proof}

When $D$ is selective, the following lemma shows we can remove the commutativity hypothesis of lemma \ref{noncommmonotonicdivision}.

\begin{lem}\label{selmonotonicdivision}Let $D$ be a selective idempotent division semiring.  Suppose $x,y\in D$ satisfy $x^n+y^n=y^n$ for some $n\geq 1$.  Then $x+y=y$.\end{lem}
\begin{proof}The lemma is clear if $x=0$.  Let $n$ be the smallest number satisfying the hypotheses of the lemma.  Suppose $x+y\neq y$.  Then $x+y=x$ since D is selective.  Note that $xy^{n-1}=(x+y)y^{n-1}=xy^{n-1}+y^n=xy^{n-1}+x^n+y^n$.  Consequently $xy^{n-1}+x^n=xy^{n-1}$.  Dividing by $x$ gives $y^{n-1}+x^{n-1}=y^{n-1}$, contradicting minimality.  Thus $x+y=y$.
\end{proof}

\begin{thm}\label{divonecyclic}Let $D$ be a division semialgebra over $\mathbb{Z}_\mathrm{max}$ with finite unit index.  Let $G=D^\times/\mathbb{Z}_\mathrm{max}^\times$.  Then $G$ has at most one cyclic subgroup of each order.\end{thm}

\begin{proof}Let $C\subseteq G$ be a cyclic subgroup of order $n$.  Let $g$ generate $C$.  Let $\hat{g}\in D^\times$ be in the preimage of $g$.  Then $\hat{g}^n\in \mathbb{Z}_\mathrm{max}^\times$.  Let $u$ denote the standard generator of $\mathbb{Z}_\mathrm{max}^\times$ as in remark \ref{elementu}.  Then there exists $k$ such that $\hat{g}^n=u^k$.

Let $d=\gcd(n,k)$.  Then $(\hat{g}^{n/d})^d=(u^{k/d})^d$.  Since $u^{k/d}$ is central, $(u^{k/d}\hat{g}^{n/d})^d=1$.  Hence $u^{k/d}\hat{g}^{n/d}=1$.  By looking at the image in $G$, we get $g^{n/d}=1$.  Since $g$ has order $n$, $d=1$.

There exist integers $a,b$ such that $an+bk=1$.  Let $g'=g^b$; note that $g'$ also generates $C$ since $\gcd(b,n)=1$.  Let $\hat{g}'=u^a\hat{g}^b$, which is a lift of $g'$.  Then $\hat{g}'^n=u^{an}\hat{g}^{bn}=u^{an}u^{bk}=u$.

Let $H\subseteq G$ be another cyclic subgroup of order $n$.  For any generator $h\in H$, the above argument gives us a new generator $h'\in H$ and a lift $\hat{h}'\in D^\times$ such that $\hat{h}'^n=u$.

Since $\hat{g}'^n=\hat{h}'^n$, we have $\hat{g}'^n=\hat{g}'^n+\hat{h}'^n=\hat{h}'^n$.  By lemma \ref{selmonotonicdivision}, $\hat{g}'=\hat{g}'+\hat{h}'=\hat{h}'$.  Projecting down to $G$ gives $g'=h'$.  Hence $C=H$.\end{proof}

\begin{cor}Let $D$ and $G$ be as in theorem \ref{divonecyclic}.  Then $G$ is cyclic.\end{cor}

\begin{thm}\label{ncclassificationui}Let $D$ be a division semialgebra over $\mathbb{Z}_\mathrm{max}$ with finite unit index.   Then $D=F^{(n)}$ for some $n$.\end{thm}
\begin{proof} Let $G=D^\times/\mathbb{Z}_\mathrm{max}^\times$.  Then $G$ is cyclic.  Since the quotient of $D^\times$ by a central subgroup is abelian, $D^\times$ is itself abelian.  Apply theorem \ref{classificationui}.\end{proof}

\section{Finite division semialgebras over $\mathbb{Z}_\mathrm{max}$}\label{fdsasect}
As before, we do not assume semirings to be commutative.

\begin{defn}Let $K$ be an idempotent semifield.  A division semialgebra $L$ over an idempotent semifield $K$ is called archimedean if for all $x\in L$, there exists $y\in K$ such that $x+y=y$.\end{defn}

\begin{thm}\label{ncarchui}Let $D$ be a finite archimedean division semialgebra over $\mathbb{Z}_\mathrm{max}$.  Then $\mathrm{ui}(D/\mathbb{Z}_\mathrm{max})<\infty$.\end{thm}\begin{proof}The reader may verify that the commutative law was never used\footnote{However the fact that $\mathbb{Z}_\mathrm{max}$ lies in the center of $D$ was used frequently.} in the proof of theorem \ref{archui}, or any of the results leading up to it.\end{proof}

\begin{defn}Let $D$ be an idempotent division semiring.  A division subsemiring $E\subseteq D$ is called convex if for any $x\in D$ such that there exist $y,z\in E$ with $x+y=y$ and $x+z=x$, one has $x\in E$.  $E\subseteq D$ is called normal if $E^\times\subseteq D^\times$ is normal.\end{defn}

\begin{thm}Let $D$ be an idempotent division semiring and $E\subseteq D$ a convex normal division subsemiring.  Then $D/E^\times$ is an idempotent division semiring.\end{thm}
\begin{proof}The fact that addition is well defined does not require the multiplicative structure, so can be proven the same way as the commutative case was in theorem \ref{convexquotient}.  Multiplication is well defined because it is well defined in $D^\times/E^\times$.\end{proof}

Substituting the above theorem and theorem \ref{ncBalgclosed} into the proof of theorem \ref{convexthm} gives the following.

\begin{thm}Let $D,E$ be finite division semirings with $E\subseteq D$ normal and convex. Suppose $D$ is finite as a left $E$-semimodule.  Then $D=E$.\end{thm}

Since section \ref{maxarchsect} never used the commutative law, we have the following.
\begin{thm}Let $D$ be an idempotent division semiring.  There is a maximal archimedean division subsemiring $D_\mathrm{arch}\subseteq D$.  Furthermore $D_\mathrm{arch}\subseteq D$ is convex and normal.\end{thm}

\begin{proof}We only need to show normality.  Let $x\in D_\mathrm{arch}^\times$ and $g\in D^\times$.  Then by the construction of $D_\mathrm{arch}$\footnote{The construction is essentially definition \ref{maxarchdef} with $L$ replaced by $D$.}, we have $y,z\in K$ such that $x+y=y$ and $x+z=x$.  Then we get $gxg^{-1}+gyg^{-1}=gyg^{-1}$, and a similar formula involving $z$.  But $y$ and $z$ lie in $K$ which is contained in the center of $D$, so we have $gxg^{-1}+y=y$ and $gxg^{-1}+z=gxg^{-1}$.  Thus by the construction of $D_\mathrm{arch}$ we have $gxg^{-1}\in D_\mathrm{arch}$.
\end{proof}

As in section \ref{maxarchsect}, we may combine the above results to obtain the following.

\begin{cor}Every finite division semialgebra over an idempotent semifield is archimedean.\end{cor}

\begin{thm}Let $D$ be a finite division semialgebra over $\mathbb{Z}_\mathrm{max}$.  Then $D=F^{(n)}$ for some $n$.\end{thm}\begin{proof}
Since $D$ is finite over $\mathbb{Z}_\mathrm{max}$, it is archimedean over $\mathbb{Z}_\mathrm{max}$.  Since it is finite and archimedean, it has finite unit index.  We may now apply theorem \ref{ncclassificationui}.
\end{proof}

\section*{Acknowledgements}
This project was suggested by Caterina Consani, and was funded by the Johns Hopkins University.

\end{document}